\documentclass{amsart}
\usepackage[english]{babel}
\usepackage{graphicx}
\usepackage{amsmath}
\usepackage{amssymb}
\usepackage{xcolor}
\usepackage{hyperref}
\usepackage{enumerate}
\newtheorem{thm}{Theorem}[section]
\newtheorem{defn}[thm]{Definition}
\newtheorem{cor}[thm]{Corollary}
\newtheorem{lem}[thm]{Lemma}
\newtheorem{prop}[thm]{Proposition}
\theoremstyle{definition}
\theoremstyle{remark}
\newtheorem{rem}[thm]{Remark}

\numberwithin{equation}{section}

\newcommand{\D}{\mathbb D}
\newcommand{\C}{\mathbb C}
\newcommand{\N}{\mathbb N}
\newcommand{\T}{\mathbb T}

\renewcommand{\H}{\mathcal H}

\newcommand{\ba}{\begin{eqnarray*}}
\newcommand{\ea}{\end{eqnarray*}}
\newcommand{\beq}{\begin{equation}}
\newcommand{\eeq}{\end{equation}}

\begin{document}

\title{Ces\`aro-type operators on mixed norm spaces}%
\author{Oscar Blasco}%
\address{Departamento de An\'alisis Matem\'atico\\Universitat de Val\`encia\\Burjassot 46100, Valencia (Spain)}%
\email{oscar.blasco@uv.es}%

\author{Alejandro Mas}
\address{Departamento de Matemáticas,
Universidad de Alicante, San Vicente del Raspeig 03690, Alicante (Spain)}
\email{a.mas@ua.es}

\thanks{The first author is partially supported by the Spanish Project PID2022-138342NB-I00. The second author is partially supported by Ministerio de Ciencia e Innovaci\'on, Spain, project PID2022-136619NB-I00 }%


\def\N{\mathbb{N}}
\def\Z{\mathbb{Z}}
\def\K{\mathbb{K}}
\def\R{\mathbb{R}}
\def\D{\mathbb{D}}
\def\C{\mathbb{C}}
\def\T{\mathbb{T}}

\def\B{\mathcal B}

\begin{abstract}

 Given a positive Borel measure $\mu$ on $[0,1)$ and a parameter $\beta>0$, we consider the Ces\`aro-type operator $\mathcal C_{\mu,\beta}$ acting on the analytic function $f(z)=\sum_{n=0}^\infty a_n z^n$ on the unit disc of the complex plane $\mathbb D$, defined by
\[
\mathcal C_{\mu,\beta}(f)(z)= \sum_{n=0}^\infty \mu_n \left( \sum_{k=0}^n \frac{\Gamma(n-k+\beta)}{(n-k)! \Gamma(\beta)} a_k \right) z^n = \int_0^1 \frac{f(tz)}{(1-tz)^\beta} d\mu(t),
\]
where $\mu_n=\int_0^1 t^n d\mu(t)$. This operator generalizes the classical Ces\`aro operator (corresponding to the case where $\mu$ is the Lebesgue measure and $\beta=1$) and includes other relevant cases previously studied in the literature. In this paper we study the boundedness of $\mathcal C_{\mu,\beta}$ on mixed norm spaces $H(p,q,\gamma)$ for $0<p,q\leq\infty$ and $\gamma>0$. Our results extend and unify several known characterizations for the boundedness of Ces\`aro-type operators acting on spaces of analytic functions.

\end{abstract}

\maketitle

\section{Introduction}

In this paper we consider the  Ces\`aro-type operator defined by means of a positive Borel measure $\mu$ defined on $[0,1)$ and a parameter $\beta>0$ acting on spaces of analytic functions on the unit disc as follows: Given $f\in \mathcal H(\D)$, say  $f(z)=\sum_{n=0}^\infty a_nz^n$, we write 
\beq \label{cesaromubeta}\mathcal C_{\mu,\beta}(f)(z)= \sum_{n=0}^\infty \mu_n (\sum_{k=0}^n \frac{\Gamma(n-k+\beta)}{(n-k)!\Gamma(\beta)} a_k)z^n=\int_0^1 \frac{f(tz)}{(1-tz)^\beta}d\mu(t)\eeq
where $\mu_n=\int_0^1 t^n d\mu(t)$. Clearly  $\mathcal C_{\mu,\beta}(f)\in \mathcal H(\D)$ for any $f\in \mathcal H(\D)$. We
 will analyze the boundedness of $\mathcal C_{\mu,\beta}$ acting on different mixed norm spaces $H(p,q,\gamma)$ for $0<p,q\le \infty$ and $\gamma>0$, where $H(p,q,\gamma)$ consists of those analytic functions on $\D$ satisfying the condition
$$\|f\|_{(p,q,\gamma)}=(\int_0^1 (1-r)^{\gamma q-1} M^q_p(f,r) dr)^{1/q}<\infty$$
where as usual $M_p(f,r)=(\int_0^{2\pi}|f(re^{i\theta})|^p\frac{d\theta}{2\pi})^{1/p}=\|f_r\|_{H^p}$ with $f_r(z)=f(rz)$.

The case $d\mu(t)=dt$ and $\beta=1$ corresponds to the classical
Ces\`{a}ro operator, denoted by $\mathcal C$, that is 
\beq \label{cesaro}
\mathcal C(f)(z)=\sum_{n =0}^\infty \frac{1}{n+1}(\sum_{k=0}^n a_k) z^n=\int_0^1 \frac{f(tz)}{1-tz}dt\eeq
for any $f\in \H(\D)$.

The boundedness of the Ces\`aro operator $\mathcal C$ on Hardy spaces $H^p$ for $0 < p < \infty$, weighted Bergman spaces $A^p_\alpha$ for $0<p<\infty$ and $\alpha>-1$ and mixed norm spaces $H(p,q,\gamma)$ for $0<p,q<\infty$ and $\gamma>0$ has been established by various authors using different approaches (see for instance \cite{A, H, M, S1,S2,S3, St}).

 Andersen in \cite{A} considered the case $d\mu_\beta(t)= \beta(1-t)^{\beta-1}dt$ and $\beta>0$, denoting the associated operator by
$$\mathcal C^{\beta-1}(f)(z)= \beta\int_0^1\frac{f(tz)(1-t)^{\beta-1}}{(1-tz)^{\beta}}dt=\sum_{n=0}^\infty \frac{1}{A_n^{\beta}} (\sum_{k=0}^n A_{n-k}^{\beta-1} a_k)z^n$$ where  $f(z)=\sum_{n=0}^\infty a_nz^n$ and
$A_n^\alpha=\frac{\Gamma(n+1+\alpha)}{n!\Gamma(1+\alpha)}$ for $\alpha>-1$.

He showed the boundedness of $\mathcal C^{\beta-1}$ on $H(p,q,\gamma)$ for $0<p,q<\infty$ and $\gamma>0$.

More recently Galanopoulos, Girela and Merch\'{a}n in  \cite{GGM} dealt with the case $\beta=1$  and denoted \beq \label{cesaromu} \mathcal C_{\mu}(f)(z)= \sum_{n=0}^\infty \mu_n (\sum_{k=0}^n  a_k)z^n=\int_0^1 \frac{f(tz)}{1-tz}d\mu(t),\eeq where $\mu_n=\int_0^1 t^n d\mu(t)$ and $f(z)=\sum_{n=0}^\infty a_nz^n$.

In \cite{GGM}, the authors showed that  the boundedness of $\mathcal C_\mu$ on Hardy spaces and weighted Bergman spaces holds only for Carleson measures $\mu$, where in this situation, means the existence of  a given constant $C>0$ such that $\mu([r,1))\le C(1-r)$ for $0<r<1$ or equivalently the condition $\mu_n=O(\frac{1}{n+1})$.

Since the introduction of the operator $\mathcal C_\mu$,  numerous authors have investigated its boundedness on many other spaces of analytic functions; see, for instance, \cite{BSW,GGMM,GGM, JT}. Also, the operator $\mathcal C_\mu$ was first extended  by Blasco in \cite{B24},  by considering complex Borel measures on $[0,1)$ instead of nonnegative ones, and later by Galanopoulos, Girela, and Merchan in \cite{GGM2}, by considering complex Borel measures $\eta$ defined on the unit disc $\D$ and  $\mathcal C_{\eta}f(z)=\sum_{n=0}^\infty \eta_n(\sum_{k=0}^n a_k)z^n$ with $\eta_n=\int_{\mathbb D} w^nd\eta(w)$. The reader is referred to
\cite{BGSW,B24bis, GGM2, LX} for results in this more general setting for different spaces of analytic functions. There is still a more general formulation $\mathcal C_{(\lambda_n)}f(z)=\sum_{n=0}^\infty \lambda_n(\sum_{k=0}^n a_k)z^n$ where $(\lambda_n)$ is a sequence of complex numbers. This formulation does not have an integral representation and is also known as a Rhaly matrix operator (\cite{R1, R2}). These more general operators when acting on certain spaces of analytic functions have  recently been considered, for instance, in \cite{BGSW, LX}.

Conditions on $\mu$ and $\beta$ for the boundedness of the operator $\mathcal C_{\mu,\beta}$ for $\beta>0$ 
 have been studied by several authors, for example in \cite{JT} between different Dirichlet-type spaces, in \cite{GSZ}  for different weighted Bergman spaces and in \cite{GTZ} from the Bloch space $\mathcal{B}$ into the Bergman space $A^p$. 
 In most of the cases its boundedness is related to the fact that $\mu$ is an $s$-Carleson measure for certain value $s$, meaning $\mu([r,1))\le C (1-r)^s$ for all $0<r<1$ and some $C>0$.

In this paper we shall recover many previous results using a different approach. 
In \cite[Theorem 6]{GGM}  it was shown that $\mathcal C_\mu=\mathcal C_{\mu,1}$ maps $A^p_\alpha$ into itself for $1<p<\infty$ if and only if $\mu$ is a $1$-Carleson measure. We shall see in Theorem \ref{main2}  that such a result extends not only to the cases $p=1$ and $p=\infty$ but also to any mixed norm space $H(p,q,\gamma)$.

Recently in \cite[Theorem 2]{GSZ}  it has been shown, making use of a generalized Schur’s test, that  $\mathcal C_{\mu,\beta}$ maps $A^p_\alpha$ into $A^q_\alpha$ for $1\le p\le q<\infty$ if and only if $\mu$ is an $s$-Carleson measure where
$s=\beta+ (2+\alpha)(\frac{1}{p}-\frac{1}{q}).$ We shall recover such a result from our results using embeddings between mixed norm spaces. 

Our technique will be to look at the function $\mathcal C_\mu(1)(z)=F_\mu(z)=\int_0^1 \frac{d\mu(t)}{1-tz}$ and to consider 
the operator $\mathcal C_{\mu,\beta}$  as a composition of two operators: either a  Hadamard multiplier with symbol $F_\mu$ and the multiplication operator with symbol $K_{\beta-1}$ or a  Hadamard multiplier with symbol the $\beta$-fractional derivative of the function $F_\mu$ and the weighted Ces\`aro operator $\mathcal C^{\beta-1} $, namely
$$ \mathcal C_{\mu,\beta}f= F_\mu* fK_{\beta-1}= D_{\beta} F_\mu* \mathcal C^{\beta-1}f$$
where, for each $\alpha>-1$, we denote  $K_{\alpha} (z)=\frac{1}{(1-z)^{\alpha+1}}= \sum_{n=0}^\infty \frac{\Gamma(n+\alpha+1)}{n!\Gamma(\alpha+1)}  z^n$,  the Hadamard product of two functions $f, g\in \mathcal H(\D)$ is given by
$$  f*g(z)=\sum_{n=0}^\infty a_nb_n z^n$$
for $f(z)=\sum_{n=0}^\infty a_nz^n$ and $g(z)=\sum_{n=0}^\infty b_nz^n$
and the fractional derivative is defined by 
$$D_{\alpha} f(z)= f*K_{\alpha} (z)= \sum_{n=0}^\infty \frac{\Gamma(n+\alpha+1)}{n!\Gamma(\alpha+1)} a_n z^n.$$

We shall analyze the boundedness of $\mathcal C_{\mu,\beta}:H(p_1,q_1,\gamma_1)\to
H(p_2,q_2,\gamma_2)$ for different values of the parameters $p,q$ and $\gamma$. This will allow, among other things, to cover the study of their boundedness from  $A^p_{\alpha_1}$  to $A^p_{\alpha_2}$ for $\alpha_1\ne \alpha_2$ or from $A^p_\alpha$ to $A^q_\alpha$ for $p<q$ that had been previously considered.

Moreover, we shall see that the Carleson-type conditions on $\mu$ that have appeared previously in many papers are reformulations of the fact that certain fractional derivative of $F_\mu$ belongs to a mixed norm space $H(p,\infty, \gamma)$ (see Theorem \ref{carlesonfmu}).

The paper is divided into seven sections. Sections 2 and 3 are of a preliminary nature. In them, we introduce mixed norm spaces $H(p,q,\gamma)$ and  fractional derivatives $D_\alpha$, respectively and present some properties to be used later on. In Section 4 we introduce the fundamental function $F_\mu$ and show how to describe $D_\alpha F_\mu\in H(p,q,\gamma)$ in terms of the moments $(\mu_n)$. Section 5 contains some preliminaries on  Carleson measures and their connection with $D_\alpha F_\mu$ (see Theorem \ref{carlesonfmu}).
The main results are in Section 6, where we analyze the boundedness of $\mathcal C_{\mu,\beta}$ acting between different spaces $H(p,q,\gamma)$.
In our main result, we shall notice that the $s$-Carleson condition on a measure is actually equivalent to the fact that $\mathcal C_{\mu,\beta}$ maps $H(p,q,\gamma_1)$ to $H(p,q,\gamma_2)$ for $s=\beta+\gamma_1-\gamma_2$ (see Theorem \ref{main2}). As a consequence in Corollary \ref{corCmubetaBS} we recover \cite[Theorem 2]{GSZ}.  

Finally, in Section 7, we manage to get some additional conditions on the measure for the operator $\mathcal C_{\mu,\beta}$ to map $H(p,q_1,\gamma)$ into $H(p,q_2,\gamma)$ for $q_2<q_1$. These extra conditions are completely described for $q_1=\infty$ (see Theorem \ref{teofinal}) in terms of the fact that certain fractional derivative of $F_\mu$ belongs to the range space $H(p,q_2,\gamma)$.  In particular, in Corollary \ref{corofinal2}, we show that in the case $p\ge 2$ the boundedness of $\mathcal C_\mu$ from $H(p,\infty,\gamma)$ into $H(p,q,\gamma)$ is actually equivalent to $(\mu_n (n+1)^{1-1/q})\in \ell^q$.

Throughout the paper the letter $C=C(\cdot)$ will denote an absolute constant whose value depends on the parameters indicated
in the parenthesis, and may change from one occurrence to another.
We will use the notation $a\lesssim b$ if there exists a constant
$C=C(\cdot)>0$ such that $a\le Cb$, and $a\gtrsim b$ is understood
in an analogous manner. In particular, we write $a\approx b$ and say that $a$ and $b$ are comparable if $a\lesssim b$ and
$a\gtrsim b$.

\section{Preliminaries on mixed norm spaces}

For $0<q<\infty$, $0<p\leq \infty$ and $\gamma>0$ we denote $H(p,q,\gamma)$ the mixed norm space of analytic functions in the unit disc satisfying the condition
$$\|f\|_{(p,q,\gamma)}=(\int_0^1 (1-r)^{\gamma q-1} M^q_p(f,r) dr)^{1/q}<\infty.$$
Similarly we denote $H(p,\infty, \gamma)$ the space of analytic functions in the unit disc such that $$ \|f\|_{(p,\infty,\gamma)}= \sup_{0<r<1}  (1-r)^\gamma M_p(f,r)<\infty.$$
Since $M_p(f,r)$ is increasing in $r$ we sometimes will use the fact that
\beq \label{equiv}
\|f\|_{(p,q,\gamma)}\approx (\int_0^1 (1-r)^{\gamma q-1} M^q_p(f,r^2) dr)^{1/q}.
\eeq

With this scale of spaces, we recover some classical ones, for instance  Korenblum spaces $A^\infty_\gamma$, consisting of analytic functions satisfying $|f(z)|=O(\frac{1}{(1-|z|)^\gamma})$, corresponding to $H(\infty,\infty,\gamma)$ or the weighted Bergman spaces $A^p_\alpha$ for $\alpha>-1$, which are spaces of analytic functions for which $ \int_\D |f(z)|^p(1-|z|^2)^{\alpha}dA(z)<\infty$ where $dA(z)$ denotes the normalized Lebesgue measure on the unit disc,  corresponding to $H(p,p, \frac{1+\alpha}{p})$.

The following inclusions are well known and easy to show: \begin{equation} \label{inclugamma}H(p,q,\gamma_1)\subset H(p,q,\gamma_2), \quad \gamma_1\le \gamma_2, \end{equation}
\beq \label{inclup} H(p_2,q,\gamma)\subset H(p_1,q,\gamma), \quad p_1\le p_2,\eeq
\beq \label{incluq}H(p,q_1,\gamma)\subset H(p,q_2,\gamma), \quad q_1\le q_2.\eeq
Let us recall that Hardy-Littlewood theorem (\cite[Theorem 5.11]{D}) gives that
\beq \label{hl} H^p\subset H(q,p,1/p-1/q), \quad p<q .\eeq
From (\ref{hl}) one easily gets
\begin{equation} \label{inclu}H(p_2,q,\gamma)\subset H(p_1,q,\gamma)\subset  H(p_2,q,\gamma +1/p_1-1/p_2) , \quad p_1<p_2.\end{equation}
In particular we shall use later on the following inclusion
\begin{equation} \label{incluBergman}A^p_\alpha\subset H(q,p,(\alpha+2)/p-1/q), \quad  p\le q. \eeq

Let us mention two simple facts to be used in the sequel.
\begin{lem} \label{lemaprod} Let $f\in H(p_1,q_1,\alpha_1)$ and $g\in H(p_2,q_2,\alpha_2)$. Then $$fg\in H(p_3,q_3,\alpha_3)$$ where
$\alpha_3=\alpha_1+\alpha_2$, $\frac{1}{p_3}=\frac{1}{p_1}+\frac{1}{p_2}$ and $\frac{1}{q_3}=\frac{1}{q_1}+\frac{1}{q_2}$.
\end{lem}

\begin{proof}
Recall that $f\in H(p_1,q_1,\alpha_1)$ means $(1-r)^{\alpha_1}M_{p_1}(f,r)\in L^{q_1}(\frac{dr}{1-r})$.
Hence H\"older's inequality in the parameter $p$ gives
$$ (1-r)^{\alpha_3}M_{p_3}(fg,r)\le (1-r)^{\alpha_1}M_{p_1}(f,r) (1-r)^{\alpha_2}M_{p_2}(g,r) $$
and then H\"older's inequality in the parameter $q$ gives the desired result.
\end{proof}

\begin{lem} \label{lemacon} Let $1\le p_1,p_2\le \infty$,  $1/p_1+1/p_2\ge 1$ and let $1/p_3=1/p_1+1/p_2-1$, $1/q_3=1/q_1+1/q_2$ and $\gamma_3=\gamma_1+\gamma_2. $

If
$f\in H(p_1,q_1,\gamma_1)$ and $g\in H(p_2,q_2,\gamma_2)$ then $f*g\in H(p_3,q_3,\gamma_3)$.
\end{lem}
\begin{proof} By applying Young's convolution inequality,
$$
M_{p_3}(f*g, r^2)\le M_{p_1}(f, r)M_{p_2}(g, r),
$$
together with H\"older's inequality in  $L^{q}(\frac{dr}{1-r})$-spaces, we obtain the following estimate
$$\|(1-r)^{\gamma_1}M_{p_1}(f,r) (1-r)^{\gamma_2}M_{p_2}(g,r)\|_{L^{q_3}(\frac{dr}{1-r})}\le \|f\|_{(p_1,q_1,\gamma_1)}\|g\|_{(p_2,q_2,\gamma_2)}.$$
Hence, $\|f*g\|_{(p_3,q_3,\gamma_3)}\lesssim \|f\|_{(p_1,q_1,\gamma_1)}\|g\|_{(p_2,q_2,\gamma_2)} $.
\end{proof}

\section{Preliminaries on  fractional derivatives}

Let  $\beta>0$ and $\gamma\ge 0$, we shall use the notation \beq \label{nucleo1}K_{\beta-1}(z)=\frac{1}{(1-z)^\beta} = \sum_{n=0}^\infty \frac{\Gamma(n+\beta)}{n!\Gamma(\beta)} z^n\eeq
and
\beq \label{nucleo2} G_{\gamma}(z)= \sum_{n=0}^\infty \frac{n!\Gamma(\gamma+1)}{\Gamma(n+\gamma+1)} z^n.\eeq

In particular $K_0(z)= G_0(z)=\frac{1}{1-z}$. Observe that
for $\gamma >0$ we have $$G_\gamma(z)=\gamma \int_0^1\frac{(1-t)^{\gamma-1}}{1-tz}dt= \gamma \int_0^1 (1-t)^{\gamma-1}K_0(tz)dt.$$

\begin{lem}\label{fact} For $\beta, \gamma>0$ and $0< p,q<\infty$.
\beq \label{hpqK} K_{\beta-1}\in H(p,q,\gamma) \Longleftrightarrow \beta<\gamma+1/p \eeq
and
\beq \label{hpinftyK} K_{\beta-1}\in H(p,\infty,\gamma) \Longleftrightarrow \beta\le \gamma+1/p. \eeq

\end{lem}
\begin{proof}   Let us mention first the well known facts (see \cite[Theorem 1.7]{HKZ}) \begin{equation} \label{e1}
K_{\beta-1}\in H^p, \quad \beta<1/p,
\end{equation}

\begin{equation} \label{e2} M_p(K_{\beta-1},r) =O(\log^{\frac{1}{p}}(\frac{1}{1-r})), \quad \beta=1/p,
 \end{equation}
\begin{equation} \label{e3}
M_p(K_{\beta-1}, r)= O(\frac{1}{(1-r)^{\beta-1/p}}), \quad \beta>1/p.
\end{equation} From (\ref{e1})  we have that  $K_{\beta-1}\in H^p\subset  H(p,q,\gamma)$ for  $\beta<1/p$ and any $\gamma>0$.

Similarly using (\ref{e2})  we have
$\int_0^1 (1-r)^{\gamma q-1}(\log(\frac{1}{1-r}))^{\frac{q}{p}} dr<\infty$ for $\beta=1/p$ and  any $\gamma>0$.

Now if $1/p<\beta<\gamma +1/p$ we use (\ref{e3})
to obtain $K_{\beta-1}\in H(p,q,\gamma)$  because
$\int_0^1 \frac{(1-r)^{\gamma q-1}}{(1- r)^{q\beta-q/p}} dr<\infty$ .

Assume now that $\beta \ge \gamma+1/p$ and let us show that $K_{\beta-1}\notin H(p,q,\gamma) $. It suffices to see that   $K_{\beta-1}\notin H(p,\infty,\gamma) $. Since $\beta p>1$, if  $K_{\beta-1}\in H(p,\infty,\gamma) $, we have that
$$M^p_p(K_{\beta-1},r)=\int_0^{2\pi}\frac{1}{|1-re^{i\theta}|^{\beta p}} \frac{d\theta}{2\pi}\approx \frac{1}{(1-r)^{\beta p-1}} \lesssim \frac{1}{(1-r)^{p \gamma}}.$$ This gives a contradiction if $\beta>\gamma+1/p$.  In the case  $\gamma= \beta-1/p$ we also obtain that $K_{\beta-1}\notin H(p,q,\alpha) $ since $\int_0^1 \frac{(1-r)^{\gamma q-1}}{(1-r)^{ \gamma q}} dr=\infty$.

For the case $q=\infty$ the above proof works, with the difference that $K_{\beta-1}\in H(p,\infty, \gamma)$ for $\beta-1/p=\gamma$.
\end{proof}

\begin{defn} 
For each $\gamma>-1$ we define the fractional derivative  
$$D_{\gamma} f(z)= f*K_{\gamma} (z)= \sum_{n=0}^\infty \frac{\Gamma(n+\gamma+1)}{n!\Gamma(\gamma+1)} a_n z^n$$
and  the fractional integral by
$$  I_{\gamma} f(z)= f*G_{\gamma} (z)= \sum_{n=0}^\infty \frac{n!\Gamma(\gamma+1)}{\Gamma(n+\gamma+1)}a_n z^n.
$$

With this notation  for each $f\in \mathcal H(\D)$ we have
$ I_{\gamma}D_{\gamma}f= D_{\gamma}I_{\gamma}f=f.$

In particular $D_0f=I_0 f=f$ and writing $D_1=D$ and $I_1=I$ we have $$Df(z)=  (zf)'(z) \quad \hbox{and} \quad If(z)=\frac{1}{z}\int_0^z f(s)ds.$$

Also observe that for $\gamma>0$ we have 
$$I_{\gamma} f(z)=\gamma\int_0^1 (1-t)^{\gamma-1}f(tz)dt.$$

\end{defn}

The next result is part of the folklore, and its proof can be found in \cite{B4} but we include a proof here for completeness. We shall use the following elementary lemma.
\begin{lem} \label{comp} Let $\gamma>-1$. Then \beq\label{deriv}(\gamma+1)D_{\gamma+1}f=  D_\gamma Df + \gamma D_\gamma f.\eeq

\end{lem}
\begin{proof}
Note that
$$\frac{(n+1+\gamma)\Gamma(n+1+\gamma)}{n!( \gamma+1)\Gamma(\gamma+1)}=\frac{1}{\gamma+1}\frac{(n+1)\Gamma(n+1+\gamma)}{n! \Gamma(\gamma+1)} + \frac{\gamma}{\gamma+1}\frac{\Gamma(n+1+\gamma)}{n!\Gamma( \gamma+1)}.$$
\end{proof}

\begin{lem} \label{derivMNS} \textup{(\cite[Theorem A]{B4})} Let $\gamma,\alpha>0$, $1\le p\le \infty$, $0<q\le \infty$ and $f\in \mathcal H(\D)$. Then $f\in  H(p,q,\gamma)$ if and only if $ D_\alpha f\in H(p,q,\alpha+\gamma)$.
\end{lem}

\begin{proof}  Assume that $f\in  H(p,q,\gamma)$. Since $D_\alpha f= K_{\alpha}* f$ and from (\ref{hpinftyK}) we know that $K_{\alpha}\in H(1,\infty,\alpha)$ then  $D_\alpha f\in H(p,q,\gamma+\alpha)$ using Lemma \ref{lemacon}.

  Conversely, assume that  $ D_\alpha f\in H(p,q,\alpha+\gamma)$, using that 
 $$
 f(z)= I_\alpha D_\alpha f(z)=\alpha\int_0^1 (1-t)^{\alpha-1}D_\alpha f(tz)dt
 $$
and vector-valued Minkowski's inequality we get
\beq
M_p(f,r) \le \alpha \int_0^1(1-t)^{\alpha-1}M_p(D_\alpha f,r t)dt = \alpha \| (D_\alpha f)_r\|_{(p,1,\alpha)}.
\eeq

For $q=\infty$ we  easily obtain that $f\in H(p,\infty, \gamma)$ since
$$M_p(f,r)\lesssim \int_0^1 \frac{(1-t)^{\alpha-1}}{(1-rt)^{\alpha+\gamma}}dt\lesssim  \frac{1}{(1-r)^{\gamma}}.$$

We deal first with the case $\alpha\ge 1 $.
For $0<q\le 1$  we can use that $H(p,q,\alpha)\subset H(p,1,\alpha)$ and, since $(1-t)^{\alpha-1}\le (1-rt)^{\alpha-1}$ for $0<t,r<1$, we obtain
\ba 2rM^q_p(f,r) &\le&  2r\alpha^q \| (D_\alpha f)_{r}\|^q_{(p,q,\alpha)}\\
&\lesssim& \int_0^r (1-t)^{\alpha q-1}M^q_p(D_\alpha f, t)dt.
\ea
This implies that
\ba
\|f\|^q_{(p,q,\gamma)}&\lesssim&  \int_0^1(1-r)^{\gamma q-1}(\int_0^r (1-t)^{\alpha q-1}M^q_p(D_\alpha f, t)dt)dr\\
&=&  \int_0^1(\int_t^1(1-r)^{\gamma q-1}dr)(1-t)^{\alpha q-1}M^q_p(D_\alpha f, t)dt\\
&\approx&  \|D_\alpha f\|^q_{(p,q,\alpha+\gamma)}.
\ea
For $1<q<\infty$ we denote $$A=\int_0^1 (1-r)^{\gamma q-1}(\int_0^r (1-t)^{\alpha-1}M_p(D_\alpha f, t)dt)^q dr.$$
Now using  integration by parts and H\"older's inequality we have
\ba
A&\lesssim&  \int_0^1 (1-r)^{\gamma q +\alpha-1}M_p(D_\alpha f, r)(\int_0^r (1-t)^{\alpha-1}M_p(D_\alpha f, t)dt)^{q-1} dr\\
&\lesssim& \Big(\int_0^1 (1-r)^{q(\gamma  +\alpha)-1}M^q_p(D_\alpha f, r) dr \Big)^{1/q} A^{1/q'}.
\ea

We use again the estimate
$$rM_p(f,r^2)\lesssim\int_0^r (1-t)^{\alpha-1}M_p(D_\alpha f, t)dt  $$
and then, $\|f\|_{(p,q,\gamma)}\lesssim A^{1/q}\lesssim\|D_\alpha f\|_{(p,q,\alpha+\gamma)}$.

We  deal now with $\alpha<1$. This case follows from the previous one, because  if $D_\alpha f \in H(p,q,\alpha+\gamma)$ for some $0<\alpha<1$ then, using  the direct implication we have that  $DD_\alpha f \in H(p,q,\alpha+\gamma+1)$. Now, invoking Lemma \ref{comp}, we obtain $D_{\alpha+1}f\in H(p,q,\alpha+\gamma+1)$ and finally the previous case gives that $f\in H(p,q,\gamma).$
\end{proof}

\begin{cor}  \label{coroderiv}  Let $0<\alpha_1<\gamma$,  $\alpha_2\ge 0$,  $1\le p\le \infty$ and $0<q\le \infty$. Then
$D_{\alpha_1}f\in H(p,q, \gamma)$ if and only if $ D_{\alpha_2}f \in H(p,q, \gamma-\alpha_1+\alpha_2)$.
\end{cor}

\section{ The fundamental function  \texorpdfstring{$F_\mu$}{Fmu}}
\begin{defn}  Given a positive Borel measure $\mu$ defined on $[0,1)$ we write
$$
F_\mu(z)= \sum_{n=0}^\infty \mu_n z^n= \int_0^1 \frac{d\mu(t)}{1-tz},
$$
where $\mu_n=\int_0^1 t^n d\mu(t)$.

\end{defn}
\begin{rem}
    For $\beta>0$ we denote $d\mu_\beta(t)= \beta(1-t)^{\beta-1} dt$ and
$F_{\mu_\beta}(z)= G_\beta(z).$
\end{rem}

\begin{lem} \label{derivfmu} Given a positive Borel measure $\mu$ defined on $[0,1)$ and $\alpha> -1$ then
$$D_\alpha F_\mu(z)=  \int_0^1 \frac{d\mu(t)}{(1-tz)^{\alpha+1}}=\sum_{n=0}^\infty \frac{\Gamma(n+\alpha+1)}{n!\Gamma(\alpha+1)} \mu_n z^n.$$
In particular $$D_\alpha F_\mu\in H(p,\infty, \alpha+1/p'), \quad 1\le p\le \infty, \quad \alpha >-1/p'$$ and 
$$D_\alpha F_\mu\in H(p,1, \gamma), \quad 1\le p\le \infty, \quad \gamma>\alpha +1/p'.$$
\end{lem}
\begin{proof}
Note that $D_\alpha f(z)=K_\alpha*f(z)=\int_0^{2\pi} K_\alpha (e^{i\theta}z)f(e^{-i\theta})\frac{d\theta}{2\pi}$ and then
$$ D_\alpha F_\mu(z)= \int_0^1(\int_0^{2\pi}\frac{ K_\alpha (e^{i\theta}z)}{1-te^{-i\theta}}\frac{d\theta}{2\pi})d\mu(t)=\int_0^1K_\alpha (tz)d\mu(t).$$
Now using Minkowski's inequality and $\alpha>-1/p'$
$$M_p(D_\alpha F_\mu,r)\lesssim \int_0^1 \frac{d\mu(t)}{(1-rt)^{\alpha +1/p'}}\lesssim \frac{1}{(1-r)^{\alpha+ 1/p'}}.$$
Similarly, if $\gamma>\alpha+1/p'$
$$\int_0^1 (1-r)^{\gamma-1} M_p(D_\alpha F_\mu,r)dr \lesssim \int_0^1(\int_0^1 \frac{(1-r)^{\gamma-1}}{(1-r)^{\alpha+1/p'}}dr)d\mu(t)<\infty.$$

\end{proof}

From Lemma \ref{derivfmu}  we know that for any positive Borel measure $\mu$ we always have $D_\alpha F_\mu\in A^\infty_{\alpha+1}$ for any $\alpha> -1$ and $D_\alpha F_\mu\in H(\infty,1,\gamma)$ (and hence $D_\alpha F_\mu\in H(p, q,\gamma)$) for any $p,q\ge 1$ and $\gamma>\alpha+1$.
We are interested in finding when $D_\alpha F_\mu\in A^\infty_\gamma$ for $\gamma < \alpha+1$ or $D_\alpha F_\mu\in H(p,q,\gamma)$ for $q<1$ or $\gamma\le \alpha+1$. 
 \begin{prop} \label{firstcase} Let $\gamma>0, \alpha>-1$ and  let $\mu$ be a positive Borel measure defined on $[0,1)$.

\begin{enumerate}[(i)]
    \item   Let $1\le p<\infty$, $-1<\alpha\le 0$. Then
$$ D_\alpha F_\mu\in H^p \Longleftrightarrow \sum_{n=0}^\infty \frac{\mu^p_n}{(n+1)^{2-p(1+\alpha)}}<\infty.$$

\item $D_\alpha F_\mu\in A^\infty_\gamma \Longleftrightarrow \mu_n =O( (n+1)^{\gamma-\alpha-1 }).$

\item $D_\alpha F_\mu\in H(2, \infty, \gamma)  \Longleftrightarrow  \mu_n =O( (n+1)^{\gamma-\alpha-1/2}) .$

\item Let $1\le p<\infty$ and $-1<\alpha<\gamma$. Then $$D_\alpha F_\mu\in H(p, \infty, \gamma)  \Longleftrightarrow  \mu_n =O( (n+1)^{\gamma-\alpha-1/p' })  .$$
\end{enumerate}

\end{prop}
 \begin{proof}

 (i) Since the Taylor coefficients of $D_\alpha F_\mu$ are given by the decreasing sequence $(\gamma_n)$ where $\gamma_n=\mu_n \frac{\Gamma(n+\alpha+1)}{n!\Gamma(\alpha+1)}\approx \mu_n(n+1)^\alpha$, we can use (see \cite{HL} for $1<p<\infty$ and \cite{P} for $p=1$) that \beq \label{monotoneintmeans}  \|D_\alpha F_\mu\|^p_{H^p}\approx \sum_{n=0}^\infty \frac{\gamma_n^p}{(n+1)^{2-p}}.\eeq
 This implies the result.
 
 (ii) From Lemma \ref{derivfmu} we obtain that \beq \label{minfty} M_\infty (D_\alpha F_\mu, r)=\int_0^1 \frac{d\mu(t)}{(1-rt)^{\alpha+1}}=\sum_{n=0}^\infty \frac{\Gamma(n+\alpha+1)}{n!\Gamma(\alpha+1)} \mu_n r^n.\eeq
Assume that 
$\sum_{n=0}^\infty (n+1)^\alpha \mu_n r^n\lesssim \frac{1}{(1-r)^\gamma}$.
 Hence, selecting $r= 1-\frac{1}{n+1}$ we get $$\mu_{2n}(n+1)^{\alpha+1} \lesssim \sum_{k=n}^{2n} \mu_k(k+1)^\alpha\lesssim  (n+1)^\gamma.$$
The converse is straightforward.

(iii) Use  that $$M^2_2 (D_\alpha F_\mu,r)\approx \sum_{n=0}^\infty (n+1)^{2\alpha} \mu_n^2 r^{2n}\lesssim \frac{1}{(1-r)^{2\gamma}}$$
and argue as in (ii).

(iv)  For $-1<\alpha\le 0$, due to (i) we can write
\beq \label{new}
M_p(D_\alpha F_\mu,r)\approx (\sum_{n=0}^\infty \frac{\mu^p_n r^{np}}{(n+1)^{2-p(1+\alpha)}})^{1/p}.
\eeq
Hence, $D_\alpha F_\mu \in H(p,\infty,\gamma)$ is equivalent to
 $$\sum_{n=0}^\infty \frac{\mu^p_n r^{np}}{(n+1)^{2-p(1+\alpha)}}=O(\frac{1}{(1-r)^{\gamma p}})$$
 which, arguing as in (ii) means $\mu_n=O((n+1)^{\gamma-\alpha-1/p'}).$

For $0<\alpha<\gamma$, since, due to Lemma \ref{derivMNS},  $D_\alpha F_\mu\in H(p, \infty, \gamma)$ is equivalent to $ F_\mu\in H(p, \infty, \gamma-\alpha)$ we can use the previous case for $\alpha=0$ to get the desired result.
 \end{proof}
 
 To give a description of the measures satisfying that  $D_\alpha F_\mu\in H(p,q,\gamma)$ we introduce the Kellogg spaces (see \cite{Ke}). For $0<p,q<\infty$  we denote $\ell(p,q)$ the space of sequences $(a_k)_{k=0}^\infty$ of complex numbers such that
\beq \label{kellogg}
\|(a_k)\|_{(p,q)}=\left( \sum_{n=0}^\infty ( \sum_{k\in I_n} |a_k|^p)^{q/p}\right)^{1/q}<\infty
\eeq
where $I_0= \{0\}$ and $I_n= [2^{n-1}, 2^n)\cap \N$ and the obvious modifications for $p=\infty$ and $q=\infty$.
\begin{thm}  \label{equivfmu} Let $ q, \gamma>0$, $\alpha>-1$ and let $\mu$  be a positive Borel measure defined on $[0,1)$.

\begin{enumerate}[(i)]
    \item  $D_\alpha F_\mu\in H(\infty,1,\gamma)
\Longleftrightarrow ( (n+1)^{\alpha-\gamma}\mu_n)\in \ell^1.$

\item $D_\alpha F_\mu\in H(\infty, q, \gamma)   \Longleftrightarrow
( (n+1)^{\alpha-\gamma+1-1/q}\mu_n)\in \ell^q.$

\item  Let $1\le p<\infty$ and $-1<\alpha < \gamma$.  Then the following are equivalent:
\begin{enumerate}[(a)]
    \item $D_\alpha F_\mu\in H(p, q, \gamma).$

    \item $(\mu_n (n+1)^{-2/p+(\alpha-\gamma+1)})\in \ell(p,q)$.

    \item $((n+1)^{\alpha-\gamma+1/p'-1/q}\mu_n)\in \ell^q.$
\end{enumerate}

\end{enumerate}

\end{thm}

\begin{proof}

(i) Using (\ref{minfty})  we have
\ba\|D_\alpha F_\mu\|_{(\infty,1,\gamma)}&\approx& \int_0^1 (1-r)^{\gamma-1}
(\sum_{n=0}^\infty (n+1)^\alpha \mu_n r^n)dr\\
&=& 
\sum_{n=0}^\infty (n+1)^\alpha \mu_n (\int_0^1 (1-r)^{\gamma-1}r^ndr)\\
&\approx& 
\sum_{n=0}^\infty (n+1)^{\alpha-\gamma}\mu_n.
\ea
This gives the result.

(ii)  Using (\ref{minfty}) again we have
$$
\|D_\alpha F_\mu\|^q_{(\infty, q, \gamma) }\approx  \int_0^1 (1-r)^{\gamma q-1}(\sum_{k=0}^\infty (k+1)^\alpha \mu_k r^k)^{q} dr.
$$
It is known  (see  \cite[Lemma 2.1]{B4})  that for $\gamma>0$ and $a_k\ge 0$ for all $k$
\beq \label{lpqequiv}
\int_0^1 (1-r)^{q\gamma-1}(\sum_{k=0}^\infty a_k r^k)^{q} dr \approx \|(\frac{a_k}{(k+1)^{\gamma}})\|^q_{(1,q)}
\eeq
and then we have that $\|D_\alpha F_\mu\|_{(\infty, q, \gamma) }\approx \|((k+1)^{\alpha-\gamma}\mu_k)\|_{(1,q)}$.

Now we use that $\mu_n$ is decreasing to observe that 
$$
\mu_{2^{n}} 2^{n(\alpha-\gamma+1)} \lesssim\sum_{k\in I_n} (k+1)^{\alpha-\gamma}\mu_k\lesssim \mu_{2^{n-1}} 2^{n(\alpha-\gamma+1)} 
$$
which gives that $((k+1)^{\alpha-\gamma}\mu_k)\in \ell(1,q)$ is equivalent to
$((k+1)^{\alpha-\gamma+1-1/q}\mu_k)\in \ell^q.$

(iii) (a) $\Longleftrightarrow$ (b)   For $-1<\alpha\le 0$
we use (\ref{new}) and (\ref{lpqequiv}) to write
\begin{align*}
    \|D_\alpha F_\mu\|^q_{(p, q, \gamma) }&\approx  \int_0^1 (1-r)^{\gamma q-1}(\sum_{k=0}^\infty \frac{\mu_k ^p}{(k+1)^{2-(\alpha+1) p}}r^{pk})^{q/p} dr\\
    &\approx \| (\frac{\mu_k ^p}{(k+1)^{2-(\alpha -\gamma +1) p}})\|_{(1,q/p)}^{q/p}\\
    &\approx \| (\frac{\mu_k }{(k+1)^{2/p-(\alpha -\gamma +1) }})\|_{(p,q)}^{q}.
\end{align*}

(b) $\Longleftrightarrow$ (c) Arguing as above 
$$\mu^p_{2^{n}} 2^{n((\alpha+1-\gamma) p-1)} \lesssim\sum_{k\in I_n} (k+1)^{(\alpha+1-\gamma) p-2}\mu^p_k\lesssim \mu^p_{2^{n-1}} 2^{n((\alpha+1-\gamma) p-1)}. $$
Finally notice that $((k+1)^{ (\alpha-\gamma +1) p-2}\mu^p_k)\in \ell(1,q/p)$ is equivalent to
$((k+1)^{(\alpha-\gamma +1) p-1-p/q}\mu^p_k)\in \ell^{q/p}$ or equivalent to
$((k+1)^{\alpha-\gamma+1/p'-1/q}\mu_k)\in \ell^q.$

In the case $\gamma>\alpha>0$ we  use that $D_\alpha F_\mu\in H(p,q,\gamma)$ is equivalent to
$F_\mu\in H(p,q,\gamma-\alpha)$ and apply  the previous case for $\alpha=0$.
\end{proof}

\section{ Preliminaries on Carleson measures}

Recall that  a positive Borel measure  $\nu$  defined on $\D$ is called an $s$-Carleson measure for $s>0$ if
\beq \label{carleson1}
\nu(S(\theta,h))=O(h^s), \quad \theta\in [-\pi,\pi),\, 0<h<1,
\eeq
where $S(\theta, h)=\{z=re^{it}\in \D: 0<1-r<h \, \hspace{2mm} \mbox{and}\hspace{2mm} |t-\theta|<\frac{h}{2}\}$ is the so-called Carleson box.

In the case of measures $\nu$ supported on $[0,1)$ the condition becomes easier.
A positive Borel measure $\mu$ defined on $[0,1)$ is $s$-Carleson whenever there exists $C>0$ such that
\begin{equation}\label{carlesonmu}\mu([r,1))\le C (1-r)^s, \quad r<1.
\end{equation}

An equivalent formulation in terms of the moment $\mu_n=\int_0^1 t^n d\mu(t)$ (see  \cite[Proposition 1]{CGP}) is given by

\begin{equation}\label{carlesonceoel}\mu_n=O(\frac{1}{(n+1)^s}).
\end{equation}

The $s$-Carleson condition can also be characterized as follows:
\begin{prop} Let $\mu$ be a positive Borel measure on $[0,1)$ and $s>0$.
$\mu$ is $s$-Carleson if and only if for any $\gamma>0$
\begin{equation}\label{carlesonpoisson2}\int_0^1 \frac{d\mu(t)}{(1-tr)^{s+\gamma}}=O(\frac{1}{(1-r)^\gamma}).
\end{equation}

\end{prop}
\begin{proof} Since 
$$\int_0^1 \frac{d\mu(t)}{(1-tr)^{s+\gamma}}\approx \sum_{n=0}^\infty \mu_n (n+1)^{s+\gamma-1}r^n$$
using (\ref{carlesonceoel}) we obtain
$$\int_0^1 \frac{d\mu(t)}{(1-tr)^{s+\gamma}}\lesssim
\sum_{n=0}^\infty  (n+1)^{\gamma-1}r^n\lesssim \frac{1}{(1-r)^\gamma}.$$
The converse follows using (\ref{carlesonmu}) since

 $$\frac{\mu([r,1))}{(1-r)^{s+\gamma}}\lesssim \int_r^1 \frac{d\mu(t)}{(1-tr)^{s+\gamma}}\lesssim \frac{1}{(1-r)^\gamma}.$$
\end{proof}

Another characterization of $s$-Carleson measures is the following (see \cite{BJ}):
$\mu$ is $s$-Carleson if and only if for any $\gamma>0$
\begin{equation}\label{carlesonpoisson}\int_0^1 \frac{d\mu(t)}{|1-tw|^{s+\gamma}}=O(\frac{1}{(1-|w|)^\gamma}).
\end{equation}

For values $s\ge 1$ we can also use \cite[Theorem 9.4]{D} to get that $\mu$ is $s$-Carleson if and only if  $H^p \subset L^q(\mu)$ for $s=q/p$ with $p\le q$, that is to say \beq (\int_0^1 |f(t)|^qd\mu(t))^{1/q}\lesssim \|f\|_p. \eeq

We would like to describe Carleson conditions in terms of the behaviour of $F_\mu$.

\begin{thm} \label{carlesonfmu}Let $1\le p<\infty$, $\alpha>-1$, $ \gamma, s>0$ and $\mu$ a positive Borel measure on $[0,1)$. The following are equivalent:

\begin{enumerate}[(i)]
    \item $\mu$ is an $s$-Carleson measure.
    \item $D_{\alpha} F_\mu\in H(p,\infty, \gamma )$ for any $\gamma<\alpha+1/p'$ and $s=\alpha+1/p'-\gamma$.

    \item $D_{\alpha}F_\mu\in A^\infty_{\gamma}$ for any $0<\gamma<\alpha+1$ and $s=\alpha+1-\gamma$.
\end{enumerate}
\end{thm}
\begin{proof} (i)  $\Longrightarrow$ (ii)
Assume $\mu$ is $s$-Carleson and let $\gamma<\alpha+1/p'$  and consider  $s=\alpha+1/p'-\gamma$ . Since  $D_{\alpha} F_\mu(z)= \int_0^1 \frac{d\mu(t)}{(1-tz)^{\alpha+1}}$ we can use Minkowski's inequality and  (\ref{carlesonpoisson2}) to get the estimate
$$M_p(D_{\alpha}  F_\mu,r)\lesssim \int_0^1 \frac{d\mu(t)}{(1-rt)^{\alpha+1/p'}} \approx \int_0^1 \frac{d\mu(t)}{(1-rt)^{s+\gamma}}\lesssim \frac{1}{(1-r)^{\gamma}}.$$
(ii) $\Longrightarrow$ (iii)  Let $0<\gamma<\alpha+1$ and $s=\alpha+1-\gamma$. Consider $\alpha'=\alpha +1/p$. Hence $\gamma<\alpha'+1/p'$ and $s=\alpha'+1/p'-\gamma$.  We can apply (ii) to obtain that
$$D_{\alpha'}F_\mu\in H(p,\infty, \gamma)\subset H(\infty,\infty, \gamma+1/p).$$
This gives $D_{\alpha}F_\mu \in A^\infty_{\gamma}$.

(iii) $\Longrightarrow$ (i)  This follows  combining Proposition \ref{firstcase} with (\ref{carlesonceoel}).
\end{proof}
Let us write the consequence of this result for $s=1$.
\begin{cor} \label{corofmu} $\mu$ is a $1$-Carleson measure if and only if
$D_\alpha F_\mu\in H(p,\infty,\alpha-1/p)$ for some (equivalently for any) $\alpha>1/p$ and $1\le p\le \infty$.
\end{cor}

\section{ Boundedness of \texorpdfstring{$\mathcal C_{\mu,\beta}$}{Cmubeta}  on mixed norm spaces}

Recall that the definition of the operator is given by
$$\mathcal{C}_{\mu,\beta}f(z)=\int_0^1 \frac{f(tz)}{(1-tz)^\beta} d\mu(t)$$
where $f\in  \mathcal H(\D)$, or in terms of Taylor coefficients
$$\mathcal{C}_{\mu,\beta}f(z)=\sum_{n=0}^\infty \mu_n (\sum_{k=0}^n \frac{\Gamma(n-k+\beta)}{(n-k)!\Gamma(\beta)} a_k)z^n$$
where $f(z)=\sum_{n=0}^\infty a_n z^n$ and $\mu_n=\int_0^1 t^n d\mu(t).$

It is clear from the definition that
$$\mathcal C_{\mu,\beta_1+\beta_2}(f)= \mathcal C_{\mu,\beta_2}(fK_{\beta_1-1}).$$

Our main tool in this section is the following formula
\beq \label{eq0} \mathcal C_{\mu,\beta}= F_\mu* fK_{\beta-1}
\eeq
which follows since $fK_{\beta-1}(z)=\sum_{n=0}^\infty (\sum_{k+j=n} \frac{\Gamma(k+\beta)}{k!\Gamma(\beta)}a_j)z^n$.

To study the boundedness of such an operator acting on mixed norm spaces for different values of  $\gamma$ we shall use the following result.

\begin{lem} If $\beta >0$ then we have that
$$  [D, \mathcal C_{\mu,\beta}]=\beta (\mathcal C_{\mu,\beta+1}-  \mathcal C_{\mu,\beta}).$$
Equivalently
\beq \label{mix}D\mathcal C_{\mu,\beta}f= \mathcal C_{\mu,\beta}(Df)+\beta \mathcal C_{\mu,\beta+1}(Sf) \eeq
where $Sf(z)=zf(z)$.
\end{lem}
\begin{proof}
    The proof of the first identity follows from 
    \[
    \frac{\Gamma(n-k+\beta+1)}{(n-k)! \Gamma(\beta+1)}= \frac{1}{\beta} \left(  (n+\beta+1)\frac{ \Gamma(n-k+\beta)}{(n-k)!\Gamma(\beta)}-\frac{ \Gamma(n-k+\beta)}{(n-k)!\Gamma(\beta)} (k+1)\right)
    \]
    and the expression in terms of coefficients for the operators.
    To obtain (\ref{mix}) we just observe that
    $$ \mathcal C_{\mu,\beta+1}f(z)-  \mathcal C_{\mu,\beta}f(z)=\int_0^1 \frac{tzf(tz)}{(1-tz)^{\beta+1}}d\mu(t)=\mathcal C_{\mu,\beta+1}(Sf)(z).$$
\end{proof}

\begin{prop} \label{bajando}  Assume that   $\mathcal C_{\mu,\beta}:H(p_1,q_1,\gamma_1) \to H(p_2,q_2,\gamma_2)$  is bounded for some $\gamma_1,\gamma_2>0$ and $0<p_1,p_2,q_1,q_2\le \infty$. Then  
    
 \begin{enumerate}[(i)]
     \item  $\mathcal C_{\mu,\beta+\delta}:H(p_1,q_1,\gamma_1-\delta) \to H(p_2,q_2,\gamma_2)$ is bounded for any $\gamma_1>\delta>0$.

     \item   $\mathcal C_{\mu,\beta}:H(p_1,q_1,\gamma_1-1) \to H(p_2,q_2,\gamma_2-1)$ is bounded whenever $\gamma_1, \gamma_2>1$.
 \end{enumerate}  
\end{prop}
\begin{proof} (i) Let $\gamma_1>\delta>0$ and $f\in H(p_1,q_1,\gamma_1-\delta)$. Since  $fK_{\delta-1}\in H(p_1,q_1,\gamma_1)$ then $\mathcal C_{\mu,\beta+\delta}(f)=\mathcal C_{\mu,\beta}(fK_{\delta-1})\in H(p_2,q_2,\gamma_2)$.

(ii) Let $f\in H(p_1,q_1,\gamma_1-1)$. We shall show that $D\mathcal C_{\mu,\beta} f\in H(p_2,q_2, \gamma_2)$.

Note that $Df\in H(p_1,q_1,\gamma_1)$ and $Sf\in H(p_1,q_1,\gamma_1-1)$. Therefore $\mathcal C_{\mu,\beta} Df\in H(p_2,q_2,\gamma_2)$ and using (i) also 
$\mathcal C_{\mu,\beta+1} (Sf)\in H(p_2,q_2, \gamma_2)$. Now the result follows from (\ref{mix}).

\end{proof}

We would like to find conditions on $\mu$ and $\beta$ to obtain   $\mathcal C_{\mu,\beta}\Big(H(p_1,q_1,\gamma_1)\Big)\subset H(p_2,q_2,\gamma_2)$ for different values of the parameters.
We start with the following general result.

\begin{lem} Let $1\le p_1,p_2\le \infty$, $0< q_1\le q_2\le \infty$ and $\beta >1/p_1'$,  Then 
$\mathcal C_{\mu,\beta} :H(p_1,q_1,\gamma_1) \to H(p_2, q_2, \gamma_2)$
is bounded for any positive Borel measure $\mu$ whenever $\gamma_2\ge \beta+ \gamma_1+1/p_1-1/p_2$. 
\end{lem}
\begin{proof}  Let $f\in H(p_1,q_1,\gamma_1)$. We use that  $K_{\beta-1}\in H(p'_1,\infty,\beta-1/p_1')$ and then invoking Lemma \ref{lemaprod}  we have that $fK_{\beta-1}\in H(1,q_1, \gamma_1+\beta-1/p_1')$.
On the other hand, $DF_\mu\in H(1,\infty,1)$ for any measure $\mu$. This gives, by Lemma \ref{lemacon}, that $DF_\mu * fK_{\beta-1}\in H(1,q_1,\gamma_1+\beta-1/p_1'+1)$.
Hence,
$$
\mathcal C_{\mu,\beta}f=F_\mu * fK_{\beta-1}\in H(1,q_1,\gamma_1+\beta-1/p_1').
$$
The result now follows trivially from the inclusions
$$H(1,q_1,\gamma_1+\beta-1/p_1')\subset H(p_2,q_1,\gamma_1+\beta+1/p_1-1/p_2)\subset H(p_2,q_2,\gamma_2).$$
\end{proof}

The range of the parameters in the previous result can be improved using the following result
\beq \label{eq1} \mathcal C_{\mu,\beta}= D_{\beta} F_\mu* \mathcal C^{\beta-1}f
\eeq
which follows since
$\mathcal C^{\beta-1}f=  I_\beta(fK_{\beta-1})$
or equivalently 
\beq \label{f0}
D_\beta(\mathcal C^{\beta-1} f)= fK_{\beta-1}.
\eeq

\begin{lem} \label{Ialpha}
    Let $\beta>0$. Then
    $$I_\beta (\mathcal C_{\mu,\beta}f)(z)=\int_0^1 \mathcal C^{\beta-1}f(tz)d\mu(t).$$
\end{lem}
\begin{proof}
    Notice that
    \begin{eqnarray*}
    I_\beta (\mathcal C_{\mu,\beta}f(z))&=&\beta\int_0^1 (1-s)^{\beta-1}(\int_0^1 \frac{f(stz)}{(1-stz)^\beta} d\mu(t)) ds\\
    &=&\beta\int_0^1 (\int_0^1 \frac{f(stz)}{(1-stz)^\beta }(1-s)^{\beta-1}ds ) d\mu(t)\\
    &=& \int_0^1 \mathcal C^{\beta-1} f(zt) d\mu(t).
    \end{eqnarray*}
\end{proof}

The weighted Ces\`aro operator $\mathcal C^{\beta-1}$ is  known to be bounded on $H(p,q,\gamma)$  for any $0<p,q<\infty$ (see \cite{A}). Here we give some improvement of such a result for $p\ge 1$ based on our approach using Hadamard multipliers.

\begin{thm} \label{cbeta} Let  $1\le p_2\le p_1 \le \infty$ and $\min\{\gamma_1,\beta\}>1/p_2-1/p_1$. Then
$$\mathcal C^{\beta-1}:H(p_1,q_1,\gamma_1)\to H(p_2,q_2, \gamma_2)
$$
is bounded for  $0< q_1\le q_2<\infty$ and $\gamma_2\ge \gamma_1+ 1/p_1-1/p_2. $

In particular
$\mathcal C^{\beta-1}$ maps $H(p,q,\gamma)$ into itself for  $\beta>0$ and it maps $H(p,q,\gamma)$ into $H(1,q, \gamma-1/p')$ for any $\min\{\beta,\gamma\}>1/ p'$.

\end{thm}
\begin{proof} 
Using (\ref{f0}),  we know that $D_\beta \mathcal C^{\beta-1}f= f K_{\beta-1}$.  From Lemma \ref{derivMNS} and  the inclusions between the mixed norm spaces, it suffices to show that $f K_{\beta-1}\in H(p_2,q_1,\beta+\gamma_1+1/p_1-1/p_2)$ for any $f\in H(p_1,q_1,\gamma_1)$.  Let $1\le p_3\le \infty$ such that $1/p_2=1/p_1+1/p_3$ and recall that $K_{\beta-1}\in H(p_3, \infty, \beta-1/p_3)$ whenever $\beta>1/p_3$ as shown in (\ref{hpinftyK}). The conclusion then follows by Lemma \ref{lemaprod}.

Choosing $p_1=p_2=p$ and $p_2=1$, $q_1=q_2=q$ and $\gamma_1=\gamma$ we obtain the particular cases.
\end{proof}

\begin{thm} Let $\mu$ be a positive Borel measure defined on $[0,1)$,   $1\le p_2\le p_1 \le \infty$ and $\min\{\gamma_1,\beta\}>1/p_2-1/p_1$. Then
$$\mathcal C_{\mu,\beta}:H(p_1,q_1,\gamma_1)\to H(p_2,q_2, \gamma_2)
$$
is bounded for  $1\le q_1\le q_2<\infty$ and $\gamma_2\ge \beta+\gamma_1+ 1/p_1-1/p_2. $
\end{thm}
\begin{proof}  Let $f\in H(p_1,q_1,\gamma_1)$.   Using now Lemma \ref{Ialpha} and the vector-valued Minkowski's inequality we get 
$$
\|I_\beta(\mathcal C_{\mu,\beta}f)\|_ {(p_2,q_2,\gamma_3)}\le \int_0^1 \|\mathcal C^{\beta-1}f_t\|_{(p_2,q_2,\gamma_3)}d\mu(t).
$$
Now, from Theorem \ref{cbeta}, we know that for $\gamma_3\ge \gamma_1+1/p_1-1/p_2$  
$$
\sup_{0<t<1}\|\mathcal C^{\beta-1}f_t\|_{(p_2,q_2, \gamma_3)}\lesssim \sup_{0<t<1}\|f_t\|_{(p_1,q_1, \gamma_1)}<\infty.
$$
Hence, we obtain that 
$I_\beta(\mathcal C_{\mu,\beta}f)\in H(p_2,q_2,\gamma_3)$ and therefore,
$\mathcal C_{\mu,\beta}f\in H(p_2,q_2,\gamma_3+\beta)$ and the proof is complete.
\end{proof}

We now analyze when the mapping $\mathcal C_{\mu,\beta}$ is bounded from $H(p_1,q_1,\gamma_1)$ into $H(p_2,q_2,\gamma_2)$ for $\gamma_2<\beta+\gamma_1+1/p_1-1/p_2$. We start with the case $q_1=q_2=\infty$.

\begin{prop} \label{carlesonhpinfty}
    Let  $1\le p_1, p_2\le \infty$ and $\beta>0$. Assume that
    $$s=\beta+\gamma_1-\gamma_2+\frac{1}{p_1}-\frac{1}{p_2} >0.$$
    The following are equivalent.
    
    (i) $\mathcal C_{\mu,\beta}:H(p_1,\infty,\gamma_1)\to H(p_2,\infty,\gamma_2)$ is bounded for some $1\le p_1,p_2\le\infty$.
    
    (ii) $\mu$ is $s$-Carleson.

    (iii) $\mathcal C_{\mu,\beta}:H(p_1,\infty,\gamma_1)\to H(p_2,\infty,\gamma_2)$ is bounded for all $1\le p_2\le p_1\le\infty$
    and $\beta>\frac{1}{p_2}-\frac{1}{p_1}$.
\end{prop}
\begin{proof}  (i) $\Longrightarrow$ (ii) Assume that 
$\mathcal C_{\mu,\beta}:H(p_1,\infty,\gamma_1)\to H(p_2,\infty,\gamma_2)$ for a given pair $1\le p_1,p_2\le \infty.$
Set $\alpha=\gamma_1-\frac{1}{p_1'}+\beta$.
Using that $K_{\gamma_1-\frac{1}{p_1'}}\in H(p_1,\infty,\gamma_1)$ we have that $$D_{\alpha}F_\mu=\mathcal C_{\mu,\beta}(K_{\gamma_1-\frac{1}{p_1'}})\in H(p_2,\infty,\gamma_2).$$
  Observe that 
$\gamma_2< \gamma_1+\frac{1}{p_1}-\frac{1}{p_2}+\beta=\alpha + \frac{1}{p'_2}.$ 
Invoking Theorem \ref{carlesonfmu} we conclude that  $\mu$ is $s$-Carleson for $s=\beta+\gamma_1-\gamma_2+\frac{1}{p_1}-\frac{1}{p_2}$. 

(ii) $\Longrightarrow$ (iii) Assume now that $\mu$ is $s$-Carleson. Let $1\le p_2\le p_1\le \infty$, $\frac{1}{p_2}=\frac{1}{p_1} 
+\frac{1}{p_3}$ and $\beta>\frac{1}{p_3}$. We can estimate
\begin{align*}
M_{p_2}(\mathcal C_{\mu,\beta}f,r)&\le 
\int_0^1 M_{p_2}(fK_{\beta-1},rt)d\mu(t)\\
&\le 
\int_0^1 M_{p_1}(f,rt)M_{p_3}(K_{\beta-1},rt)d\mu(t)\\
&\lesssim 
\int_0^1 \frac{M_{p_3}(K_{\beta-1},rt)}{(1-rt)^{\gamma_1}}d\mu(t)\\
&\lesssim 
\int_0^1 \frac{d\mu(t)}{(1-rt)^{\gamma_1+\beta-\frac{1}{p_3}}}\\
&\approx
\int_0^1 \frac{d\mu(t)}{(1-rt)^{s+\gamma_2}}
\lesssim \frac{1}{(1-r)^{\gamma_2}}.
\end{align*}
(iii) $\Longrightarrow$ (i) is obvious.
\end{proof}

The next result provides a consequence of the boundedness of $\mathcal C_{\mu,\beta}$ from $H(p_1,q_1,\gamma_1)$ into $H(p_2,q_2,\gamma_2)$ in terms of a Carleson condition when $0<q_1,q_2<\infty$.

\begin{lem}\label{lemasufcondcarleson}
    Let $0<p_1,p_2,q_1,q_2,\gamma_1,\gamma_2, \beta<\infty$ such that
    $$s=\beta+\gamma_1-\gamma_2+\frac{1}{p_1}-\frac{1}{p_2}>0.
    $$
 If $\mathcal C_{\mu,\beta}$ maps $H(p_1,q_1,\gamma_1)$ into $H(p_2,q_2,\gamma_2)$, then $\mu$ is an $s$-Carleson measure.

   In particular, if $\mathcal C_{\mu,\beta}$ maps $H(p,q_1,\gamma_1)$ into $H(p,q_2,\gamma_2)$ for some $0<p<\infty$, $0<q_1,q_2<\infty$ and $\gamma_2<\gamma_1+\beta$, then $\mu$ is a $(\beta+\gamma_1-\gamma_2)$-Carleson measure.
\end{lem}
\begin{proof}
    Assume that $\mathcal C_{\mu,\beta}$ maps $H(p_1,q_1,\gamma_1)$ into $H(p_2,q_2,\gamma_2)$.   Let $0<r<1$  and define  
\[
f_r(z)=\frac{1}{(1-rz)^{\frac{1}{p_1}+\frac{1}{q_1}+\gamma_1}}.
\]
 It is easy to see that  $\|f_r\|^{q_1}_{H(p_1,q_1,\gamma_1)}\lesssim \frac{1}{1-r}$.

Hence, denoting
$$g_r(z)=\int_0^1 \frac{d\mu(t)}{(1-tz)^\beta(1-trz)^{\frac{1}{p_1}+\frac{1}{q_1}+\gamma_1}}$$
we have  $\|g_r\|^{q_2}_{H(p_2,q_2,\gamma_2)}\lesssim \frac{1}{(1-r)^{q_2/q_1}}$.
 Using F\`ejer-Riesz inequality  (see \cite[Theorem 3.13]{D}) we have
$$\int_0^1 (\int_0^1 \frac{d\mu(t)}{(1-t\rho s)^{\beta}(1-tr\rho s)^{\frac{1}{p_1}+\frac{1}{q_1}+\gamma_1}})^{p_2}ds\lesssim  M_{p_2}^{p_2}(g_r,\rho), \quad 0<\rho<1.$$
Since 
$$
\int_0^1 \frac{d\mu(t)}{(1-t\rho s)^{\beta}(1-t r\rho s)^{\frac{1}{p_1}+\frac{1}{q_1}+\gamma_1}} \gtrsim  \frac{\mu([r,1))}{(1-r\rho s)^{\frac{1}{p_1}+\frac{1}{q_1}+\gamma_1+\beta}},
$$
we get 
\begin{align*}
M_{p_2}^{p_2}(g_r,\rho)&\gtrsim \mu^{p_2}([r,1))\int_0^1 \frac{ds}{(1-r\rho s)^{p_2 \left( \frac{1}{p_1}+\frac{1}{q_1}+\gamma_1+\beta\right)}} \\
&\gtrsim \frac{\mu^{p_2}([r,1))}{(1-r\rho)^{p_2 \left( \frac{1}{p_1}+\frac{1}{q_1}+\gamma_1+\beta\right)-1}},
\end{align*}
and
\begin{align*}
 \|g_r\|^{q_2}_{H(p_2,q_2,\gamma_2)}&\gtrsim \mu^{q_2}([r,1))\int_0^1   \frac{(1-\rho)^{q_2\gamma_2-1}}{(1-r\rho)^{q_2 \left( \frac{1}{p_1}-\frac{1}{p_2}+\frac{1}{q_1}+\gamma_1+\beta\right)}} d\rho\\
&  \gtrsim  \frac{\mu^{q_2}([r,1))}{(1-r)^{q_2 s+ \frac{q_2}{q_1}}}.
\end{align*}
Hence, we obtain 
\[
 \frac{\mu^{q_2}([r,1))}{(1-r)^{\frac{q_2}{q_1}+q_2 s}}\lesssim \|g_r\|^{q_2}_{H(p_2,q_2,\gamma_2)}\lesssim \frac{1}{(1-r)^{\frac{q_2}{q_1}}}. 
\] 
This gives that $\mu$ is an $s$-Carleson measure.
\end{proof}

We now investigate the implications for the boundedness of $\mathcal C_{\mu,\beta}$ under the assumption that $\mu$ is an $s$-Carleson measure.
 \begin{lem}  Let $\mu$ be an $s$-Carleson measure and let $\beta>s$. Then
$$|\mathcal C_{\mu,\beta}f(z)|\lesssim\frac{P^*f(z)}{(1-|z|)^{\beta-s}} $$
where $P^*(f)(z)=\sup_{0<t<1}|f(tz)|$ is the Poisson maximal function of $f$.

In particular if $\mu$ is an $s$-Carleson measure then $$\mathcal C_{\mu,\beta}(H(p,q,\gamma))\subset H(p,q, \gamma+\beta-s).$$

\end{lem}

\begin{proof}  Observe that using \eqref{carlesonpoisson} we can write
$$|\mathcal C_{\mu,\beta}f(z)|\le  \int_0^1 \frac{|f(tz)|}{|1-tz|^{\beta}}d\mu(t) \le \sup_{0<t<1}|f(tz)|\int_0^1 \frac{d\mu(t)}{|1-tz|^\beta}\lesssim \frac{P^*f(z)}{(1-|z|)^{\beta-s}}. $$
Using that $M_p(P^*f,r)\lesssim M_p(f,r)$ we get the conclusion for any $0<p,q<\infty$.
\end{proof}

 We are now ready to state our main theorem, which in particular extends the previous lemma to $s=\beta$.
 
\begin{thm} \label{main2} Let $\gamma_1,\gamma_2,\beta>0$ such that $\gamma_2<\gamma_1+\beta$. The following statements are equivalent.

\begin{enumerate}[(i)]
    \item  $\mathcal C_{\mu,\beta}:A^\infty_{\gamma_1}\to A^\infty_{\gamma_2}$ is bounded.

    \item  $\mu$ is a $(\beta+\gamma_1-\gamma_2)$-Carleson measure.

    \item $\mathcal C_{\mu,\beta}:H(p,\infty,{\gamma_1})\to H(p,\infty,{\gamma_2})$ is bounded for all  $1\le p\le\infty$.
    
 \item $\mathcal C_{\mu,\beta}:H(p,q,{\gamma_1})\to H(p,q,{\gamma_2})$ is bounded for all $1\le p<\infty$ and $0<q< \infty$.

    \item 
 $\mathcal C_{\mu,\beta}:H(p,q,{\gamma_1})\to H(p,q,{\gamma_2})$ is bounded for some $1\le p<\infty$ and $0<q< \infty$.

 \item 
 $\mathcal C_{\mu,\beta+\delta}:H(p,q,{\gamma_1-\delta})\to H(p,q,{\gamma_2})$ is bounded for some $1\le p<\infty$, $0<\delta<\gamma_1$ and $0<q< \infty$.
\end{enumerate}
\end{thm}

\begin{proof}

     (i) $\Longrightarrow$ (ii) It follows from a particular case of Proposition \ref{carlesonhpinfty} but we include an independent argument that works for this particular case.  Assume that $\mathcal C_{\mu,\beta}:A^\infty_{\gamma_1}\to A^\infty_{\gamma_2}$ is bounded. Then for each $n\in \N$, denoting $u_n(z)=z^n$ we have
$$\sup_{0<r<1} (1-r)^{\gamma_2} M_\infty(\mathcal C_{\mu,\beta} u_n,r)\lesssim \sup_{0<r<1} (1-r)^{\gamma_1} M_\infty(u_n,r).$$
Now use that  $$\mathcal C_{\mu,\beta} u_n= \sum_{k=n}^\infty \mu_k\frac{\Gamma(k-n+\beta)}{(k-n)!\Gamma(\beta)} u_k$$  and select $r_n=\frac{n}{n+1}$ to obtain
 \[
M_\infty (\mathcal{C}_{\mu,\beta} u_n,r_n)\gtrsim \sum_{k=n}^{2n} \mu_k \frac{\Gamma(k-n+\beta)}{(k-n)!\Gamma(\beta)}\gtrsim (n+1)^{\beta}\mu_{2n}.
    \]
Using now that $$\sup_{0<r<1} (1-r)^{\gamma_1} M_\infty(u_n,r)\approx (n+\gamma_1)^{-\gamma_1}$$ we conclude that
 $\mu_{2n}\lesssim (n+1)^{\gamma_2-\gamma_1-\beta}$, then $\mu$ is a $(\beta+\gamma_1-\gamma_2)$-Carleson measure.

(ii) $\Longrightarrow$ (iii) This follows from Proposition \ref{carlesonhpinfty}.

(iii) $\Longrightarrow$ (i) It is obvious.

(ii) $\Longrightarrow$ (iv) Assume that $\mu$ is $s$-Carleson with $s=\beta+\gamma_1-\gamma_2$.
 Let us select  $N\in \mathbb{N}$ such that $s=\beta+\gamma_1-\gamma_2<N$ and $p_1>\max \left\{1,\frac{1}{\beta}\right\}$. 
 Given  $f\in H(p,q, {\gamma_1})$ we shall show that $D_N\mathcal C_{\mu,\beta} f= D_NF_\mu * fK_{\beta-1}\in H(p,q, \gamma_2+N)$. Using that  $K_{\beta-1}\in H(p_1,\infty, \beta-1/p_1)$ we have $fK_{\beta-1}\in H(p_2,q,\gamma_1+\beta-1/p_1)$ for $1/p_2=1/p+1/p_1$. 
Finally, using Theorem \ref{carlesonfmu} with $\alpha=N$ and $s=\beta+\gamma_1-\gamma_2$ we know that $D_NF_\mu\in H(p'_1,\infty, N+\frac{1}{p_1}-s)$. Hence, since $1/p_2+1/p_1'=1/p+1$, using Lemma \ref{lemacon} we get  $D_NF_\mu* fK_{\beta-1}\in H(p,q, \gamma_2+N)$.

(iv) $\Longrightarrow$ (v) It is obvious.

(v) $\Longrightarrow$ (vi) It follows from (i) in Proposition \ref{bajando}.

(vi) $\Longrightarrow$ (ii)  It follows by Lemma \ref{lemasufcondcarleson} applied to $ \tilde \beta=\beta+\delta$, $p_1=p_2=p$, $q_1=q_2=q$, $\tilde\gamma_1=\gamma_1-\delta$ and $\tilde\gamma_2=\gamma_2$.
\end{proof}

\begin{cor} \label{manicoro} Let $\gamma,\beta>0$. The following statements are equivalent.

 \begin{enumerate}[(i)]
    \item  $\mathcal C_{\mu,\beta}:A^\infty_{\gamma}\to A^\infty_{\gamma}$ is bounded.

\item  $\mu$ is a $\beta$-Carleson measure.

\item   $\mathcal C_{\mu,\beta+\delta}:H(p,q,{\gamma-\delta})\to H(p,q,{\gamma})$ is bounded for all $1\le p\le\infty$, $0\le \delta<\gamma$ and $0<q\le \infty$.

   \item  
 $\mathcal C_{\mu,\beta+\delta}:H(p,q,{\gamma-\delta})\to H(p,q,{\gamma})$ is bounded for some $1\le p\le\infty$, $0\le \delta<\gamma$ and $0<q\le \infty$.
\end{enumerate}

\end{cor}

Let us now apply Theorem \ref{main2} to obtain some applications on weighted Bergman spaces. First, using the identity $A^p_\alpha=H(p,p,\frac{\alpha+1}{p})$, we obtain the following corollary.

\begin{cor} \label{corCmubetaBergman1} Let $1\le p<\infty$, $s,\beta>0$ and  $\alpha>-1$ with $ s >\beta-\frac{1+\alpha}{p} $. Then, $\mu$ is an $s$-Carleson measure if and only if $\mathcal C_{\mu,\beta}$ maps $A^p_{\alpha+p(s-\beta)}$ into $A^p_{\alpha}$.

In particular
$\mathcal C_{\mu,\beta}$ maps $A^p_\alpha$ into itself if and only if $\mu$ is a $\beta$-Carleson measure.
\end{cor}

As an application of Corollary \ref{corCmubetaBergman1}, we obtain the following result.

\begin{cor}\label{corCmubetaBS} Let $1\le p\le q<\infty$, $\beta>0$ and $\alpha_1, \alpha_2 >-1$ satisfying that $ \frac{\alpha_1+2}{p}-\frac{\alpha_2+2}{q}+\beta >0 $. Then 
$\mathcal C_{\mu,\beta}$ maps $A^p_{\alpha_1}$ into $A^q_{\alpha_2}$ if and only if $\mu$ is an $s$-Carleson measure where
$$s=\beta+ \frac{\alpha_1+2}{p}-\frac{\alpha_2+2}{q}.$$
\end{cor}

\begin{proof}
If we assume  that $\mu$ is an $s$-Carleson measure, using Corollary \ref{corCmubetaBergman1} with $s=\beta+ \frac{\alpha_1+2}{p}-\frac{\alpha_2+2}{q}$  we know that
\[
\mathcal C_{\mu,\beta}\,: A^q_{\alpha_2+q\left( \frac{\alpha_1+2}{p}-\frac{\alpha_2+2}{q}\right)}\to A^q_{\alpha_2}.
\]
Moreover, using the inclusions of Bergman spaces ($0< p_1\leq p_2$,  $A^{p_1}_{\alpha_1}\subset A^{p_2}_{\alpha_2}$ if and only if  $\frac{\alpha_1+2}{p_1}\leq \frac{\alpha_2+2}{p_2}$, see \cite[Theorem 69]{ZZ}) we get
\[
A^p_{\alpha_1} \subset A^q_{\alpha_2+q\left( \frac{\alpha_1+2}{p}-\frac{\alpha_2+2}{q}\right)}
\]
then, $\mathcal C_{\mu,\beta}$ maps $A^p_{\alpha_1}$ into $A^q_{\alpha_2}$.

Conversely, we can apply Lemma \ref{lemasufcondcarleson} since $\mathcal C_{\mu,\beta}$ maps $H(p,p,\gamma_1)$ into $H(q,q,\gamma_2)$ with $\gamma_1=\frac{1+\alpha_1}{p}$ and $\gamma_2=\frac{1+\alpha_2}{q}$ to obtain that $\mu$ is $s$-Carleson.
\end{proof}

If $\alpha=\alpha_1=\alpha_2$ in the above result we obtain   \cite[Theorem 2]{GSZ}.

\begin{cor}
$\mathcal C_{\mu,\beta}$ maps $A^p_\alpha$ into $A^q_\alpha$ for $1\le p\le q<\infty$ and $\alpha >-1$ if and only if $\mu$ is an $s$-Carleson measure where
$$s=\beta+ (2+\alpha)(\frac{1}{p}-\frac{1}{q}).$$
\end{cor}

Let us now try to generalize Theorem \ref{main2} for different values of the parameters.

We shall use the following result which is interesting in its own right.

\begin{lem} \label{lema1} Let $\gamma, s >0$,  $1\le p_1\le p_2$  and $\beta>\frac{1}{p_1}-\frac{1}{p_2}$. If  $\mu$ is $s$-Carleson then
\begin{equation}\label{estima1} M_{p_1}(\mathcal C_{\mu,\beta}(f),r)\lesssim \int_0^ 1\frac{(1-t)^{s-1}}{ (1-rt)^{\beta-(1/p_1-1/p_2)}} M_{p_2}(f,rt)dt.\end{equation}

\end{lem}
\begin{proof}
Let $1/{p_1}=1/p_2+1/p_3$, since $\beta p_3>1$ we get $M_{p_3}(K_{\beta-1},r)\lesssim \frac{1}{(1-r)^{\beta-1/p_3}}$. Now, we consider $P^*(f)(z)=\sup_{0<t<1}|f(tz)|$ the Poisson maximal function of $f$ and $I_k=[t_k, t_{k+1})$ where $t_k=1-2^{-k}$. Hence, we have that $\mu(I_k)\lesssim(1-t_k)^s$ and therefore, selecting $0\le \phi\in L^{p_1'}$ to attains the $L^{p_1}([0,2\pi))$ norm, we get the following chain of inequalities
\ba
M_{p_1}(\mathcal C_{\mu,\beta}(f),r)&\le & (\int_0^{2\pi}(\int_0^1 |f(rte^{i\theta})||K_{\beta-1}(rte^{i\theta})|d\mu(t))^{p_1} d\theta)^{1/p_1}\\
&\lesssim & \int_0^{2\pi}(\int_0^1 |f(rte^{i\theta})||K_{\beta-1}(rte^{i\theta})|d\mu(t)) \phi(e^{i\theta})d\theta\\
&= & \int_0^{2\pi}(\sum_{k}\int_{I_k} |f(rte^{i\theta})||K_{\beta-1}(rte^{i\theta})|d\mu(t)) \phi(e^{i\theta}) d\theta\\
&\le & \int_0^{2\pi}(\sum_{k}\sup_{t\in I_k}|f(rte^{i\theta})||K_{\beta-1}(rte^{i\theta})|\mu(I_k)) \phi(e^{i\theta})d\theta\\
&\lesssim & \sum_{k} |I_k|^s \int_0^{2\pi}\sup_{t\le t_{k+1}}|f(rte^{i\theta})||K_{\beta-1}(rte^{i\theta})| \phi(e^{i\theta})d\theta\\
&= & \sum_{k} |I_k|^s \int_0^{2\pi}|P^*(fK_{\beta-1})(rt_{k+1}e^{i\theta})|\phi(e^{i\theta}) d\theta\\
&\lesssim & \sum_{k} |I_k|^s M_{p_1}(P^*(fK_{\beta-1}),rt_{k+1}) \\
&\lesssim & \sum_{k} |I_k|^s M_{p_1}(f K_{\beta-1},rt_{k+1}) \\
&\lesssim & \sum_{k} \int_{I_{k+1}} (1-t)^{s-1} M_{p_1}(f K_{\beta-1},rt) dt \\
&\le & \sum_{k} \int_{I_{k+1}} (1-t)^{s-1} M_{p_2}(f, rt) M_{p_3}(K_{\beta-1},rt)dt \\
&\lesssim &  \int_{0}^1 \frac{(1-t)^{s-1}}{ (1-rt)^{\beta-1/p_3}} M_{p_2}(f, rt)dt. \\
\end{eqnarray*}
\end{proof}

\begin{thm} \label{h1q} Let $1\le p_2\le p_1<\infty$, $0<q_1\le q_2<\infty$, $\gamma_1<\gamma_2$ and $$s=\beta+\gamma_1-\gamma_2+1/p_1-1/p_2>0.$$ Then
$\mathcal C_{\mu,\beta}:H(p_1,q_1,\gamma_1)\to H(p_2,q_2,\gamma_2)$ is bounded if and only if $\mu$ is a $s$-Carleson measure.
\end{thm}
\begin{proof}  If we assume  $\mathcal C_{\mu,\beta}\Big(H(p_1,q_1,\gamma_1)\Big)\subseteq  H(p_2,q_2,\gamma_2)$,  applying Lemma \ref{lemasufcondcarleson},  we get that $\mu$ is an $s$-Carleson measure.

Conversely, assume that  $\mu$ is an $s$-Carleson measure. 
We invoke Lemma \ref{lema1} to have
$$M_{p_2}(\mathcal C_{\mu,\beta} f,r)\lesssim \int_0^1 \frac{(1-t)^{s-1}}{(1-rt)^{s+\gamma_2-\gamma_1}}M_{p_1}(f,rt) dt.$$

Let $0<q<\infty$ and $f\in H(p_1,q,\gamma_1)$. Since $\gamma_2 > \gamma_1$ then the $s$-Carleson condition gives
$$
M_{p_2}(\mathcal C_{\mu,\beta} f,r)\lesssim  \frac{M_{p_1}(f,r)}{(1-r)^{\gamma_2-\gamma_1}}.
$$

This shows that $(1-r)^{\gamma_2}M_{p_2}(\mathcal C_{\mu,\beta} f,r)\lesssim (1-r)^{\gamma_1}M_{p_1}(f,r)$ and integrating over $L^q(\frac{dr}{1-r})$ gives $\mathcal C_{\mu,\beta} f\in H(p_2,q, \gamma_2)$. Now, using the inclusion $H(p,q_1,\gamma)\subset H(p,q_2,\gamma)$ for $q_1\le q_2$ the result is complete.
\end{proof}

Let us now use a different argument which allows to obtain a condition less restrictive than $\gamma_2>\gamma_1$  in the above theorem.

\begin{thm} \label{h1q2} Let $1\le p_2\le p_1<\infty$, $0<q_1\le q_2<\infty$ and $\beta >1/p_2-1/p_1$.  Assume that $$0<\gamma_1-(1/p_2-1/p_1)<\gamma_2 \hbox{ and } s=\beta+\gamma_1-\gamma_2+1/p_1-1/p_2>0.$$ Then
$\mathcal C_{\mu,\beta}:H(p_1,q_1,\gamma_1)\to H(p_2,q_2,\gamma_2)$ is bounded if and only if $\mu$ is an $s$-Carleson measure.
\end{thm}
\begin{proof} Only the converse needs a proof, since the direct implication was given in Theorem \ref{h1q}.
Assume $\mu$ is $s$-Carleson. On the one hand, using Theorem \ref{carlesonfmu} we have that  $D_\beta F_\mu\in H(1,\infty, \gamma_2-\gamma_1+1/p_2-1/p_1)$. On the other hand, due to Theorem \ref{cbeta}, $\mathcal C^{\beta-1}f\in H(p_2,q_2, \gamma_1+1/p_1-1/p_2)$ for any $f\in H(p_1,q_1,\gamma_1)$ . Hence by Lemma \ref{lemacon}, $\mathcal C_{\mu,\beta}f=D_\beta F_\mu*\mathcal C^{\beta-1} f\in H(p_2,q_2,\gamma_2)$ for any $f\in H(p_1,q_1,\gamma_1).$
\end{proof}

\begin{cor} Let $p\ge 1$, $0<s<1/p$ and $\mu$ a positive Borel measure. The following are equivalent.

(i) $\mu$ is an $s$-Carleson measure.

(ii) $\mathcal C_\mu: H(p,q,\gamma)\to H(1,q,\gamma+1/p-s)$ is bounded for some $0<q,\gamma<\infty$.
\end{cor}

\section{When Carleson-Type Condition is not suitable}

We have seen  in Theorem \ref{main2} that $\mathcal C_{\mu,\beta} \Big(H(p,q_1,\gamma)\Big)\subset H(p,q_2,\gamma)$ for  $q_1\le q_2$ is actually equivalent to  the $\beta$-Carleson condition of $\mu$.
We finally analyze extra conditions on $\mu$ and $\beta$ to get the inclusion $\mathcal C_{\mu,\beta} \Big(H(p,q_1,\gamma)\Big)\subset H(p,q_2,\gamma)$ for  $q_1> q_2$. 

\begin{prop} \label{p1}
Let $1\le p\le \infty$, $\gamma>0$, $\beta>1/p'$,  $0< q<\infty$ and $\mu$ a positive Borel measure. Then
   $\mathcal C_{\mu,\beta}\Big(H(p,\infty,\gamma)\Big)\subseteq H(p,q,\gamma)$ if and only if $D_{\gamma+\beta-\frac{1}{p'}}F_\mu\in H(p,q,\gamma)$.
\end{prop}
\begin{proof} Assume that  $\mathcal C_{\mu,\beta}\Big(H(p,\infty,\gamma)\Big)\subseteq H(p,q,\gamma)$.
     Since $K_{\gamma-\frac{1}{p'}}\in H(p,\infty,\gamma) $  and 
     $\mathcal C_{\mu,\beta}( K_{\gamma-1/p'})= D_{\gamma+\beta-\frac{1}{p'}}F_\mu$ we obtain that
     $D_{\gamma+\beta-\frac{1}{p'}}F_\mu\in H(p,q,\gamma).$

     Conversely, assume  $D_{\gamma+\beta-\frac{1}{p'}}F_\mu\in H(p,q,\gamma)$ and $f\in H(p,\infty, \gamma)$. Since $fK_{\beta-1}\in H(1,\infty, \gamma+\beta-1/p')$, using Lemma \ref{lemacon}, we conclude that $D_{\gamma+\beta-\frac{1}{p'}} F_\mu* fK_{\beta-1}\in H(p,q,2\gamma+\beta-\frac{1}{p'})$ and therefore $\mathcal C_{\mu,\beta} f=F_\mu* fK_{\beta-1}\in H(p,q,\gamma)$.
\end{proof}

Note that the restriction $p'\beta> 1 $  in the above proposition does not apply when $p=1$. To conclude this section, we will show that this restriction can also be removed in other cases. We begin by studying the cases $p=2$ and $p=\infty$. The following lemma will be used in our analysis.

\begin{lem} \label{lemmaAbelSum}
    Let $\lbrace \gamma_n\rbrace_{n=1}^{\infty}, \lbrace s_n\rbrace_{n=0}^{\infty}$ be a decreasing and increasing sequence of positive numbers respectively. Assume that $s_n\le t_n$ for all $n$.
    Then
    $$\sum_{n=1}^{\infty} \gamma_n (s_n-s_{n-1})\le (t_0-s_0)\gamma_1 + \sum_{n=1}^{\infty}  \gamma_n (t_n-t_{n-1}). $$
\end{lem}
\begin{proof}
    From Abel's summation by parts twice, we obtain
    \begin{align*}
    \sum_{n=1}^{\infty} \gamma_n (s_n-s_{n-1})&=\lim_{N\to \infty} \sum_{n=1}^{N} \gamma_n (s_n-s_{n-1})\\
&    =\lim_{N\to \infty}
\gamma_{N+1}s_N-s_0\gamma_1+\sum_{n=1}^{N} s_n (\gamma_{n}-\gamma_{n+1})    \\
&\leq \lim_{N\to \infty}
\gamma_{N+1}t_N-s_0\gamma_1+\sum_{n=1}^{N} t_n(\gamma_{n}-\gamma_{n+1})    \\
&=(t_0-s_0)\gamma_1+\sum_{n=1}^{\infty} \gamma_n (t_n-t_{n-1}).
    \end{align*}
\end{proof}

\begin{prop}  \label{mainprop} Let $\mu$ be a positive Borel measure defined on $[0,1)$, $\beta,\gamma >0$ and $f\in \mathcal H(\D)$. Then, for each  $p\in\{2,\infty\}$ we have
\begin{equation}
    M_p(\mathcal C_{\mu,\beta} f,r)\lesssim  \|f\|_{(p,\infty,\gamma)}M_p(D_{\beta+\gamma-1/p'}F_\mu,r),\quad 0<r<1.
\end{equation}
    \end{prop}

\begin{proof}  The case $p=\infty$ is quite simple.
Use that
$$|\mathcal C_{\mu,\beta}f(z)|\le \|f\|_{(\infty,\infty,\gamma)} \int_0^1\frac{d\mu(t)}{(1-t|z|)^{\beta+\gamma}}=\|f\|_{(\infty,\infty,\gamma)} M_\infty(D_{\beta+\gamma-1}F_\mu,r).$$

To handle the case $p=2$ we shall use that if $f\in H(2, \infty,\gamma)$, then $g=\mathcal C^{\beta-1}f=\sum_{n=0}^\infty a_n z^n \in H(2,\infty,\gamma)$.

Since $\sum_{k=0}^{n} |a_k|^2\lesssim \|g\|^2_{(2,\infty,\gamma)}(n+1)^{2\gamma}$ for each  $g\in H(2,\infty,\gamma)$, using Lemma \ref{lemmaAbelSum} with  $\gamma_n= \mu_n^2 r^{2n}$, $s_n=(n+1)^{2\beta}\sum_{k=0}^n |a_k|^2$, $t_n=Cn^{2\gamma+2\beta}$ and the fact that $(n+1)^{2\beta}|a_n|^2 \leq s_n-s_{n-1} $ we obtain
\ba
M^2_2(\mathcal C_{\mu,\beta} f, r)&=& M^2_2(D_\beta F_\mu* \mathcal C^{\beta-1}f,  r)\\
&\lesssim & \sum_{n=0}^{\infty} (n+1)^{2\beta}\mu_n^2 |a_n|^2r^{2n} \\
&\lesssim &  \|f\|^2_{(2,\infty,\gamma)}   \sum_{n=0}^{\infty} (n+1)^{2\beta+2\gamma-1}\mu_n^2 r^{2n} \\
&\lesssim & \|f\|^2_{(2,\infty,\gamma)} M_2^2( D_{\beta+\gamma-1/2} F_\mu,r).
\ea
\end{proof}

The above proposition is stronger than the boundedness from $H(p,\infty,\gamma)$ into $H(p,q,\gamma)$ because it provides control in terms of the integral means. From this result, we directly obtain the following corollary.

\begin{cor} \label{corfinal} Let $\mu$ be a positive Borel measure defined on $[0,1)$, $q,\beta, \gamma>0$ and $p\in\{2,\infty\}$. Then $\mathcal C_{\mu,\beta}$ maps $H(p,\infty,\gamma)$ into $H(p,q,\gamma)$ if and only if $D_{\beta+\gamma-1/p'}F_\mu\in H(p,q,\gamma)$.
\end{cor}

\begin{thm} \label{teofinal} Let $\mu$ be a positive Borel measure defined on $[0,1)$, $\gamma>0$,  $1\le p\le \infty$ and $0<q<\infty$. Then $\mathcal C_{\mu,\beta}$ maps $H(p,\infty,\gamma)$ into $H(p,q,\gamma)$ if and only if $D_{\beta+\gamma-1/p'}F_\mu\in H(p,q,\gamma)$.
\end{thm}
\begin{proof}  The direct implication is contained in Proposition \ref{p1} since the restriction $\beta>1/p'$ was used only for the converse.

Assume now that $D_{\beta+\gamma-1/p'}F_\mu\in H(p,q,\gamma)$.
The cases  $p\in \{2,\infty\}$ and $p'>1/\beta$ follow from Proposition \ref{p1} and Corollary \ref{corfinal}.

For each $f\in H(p,\infty,\gamma)$ we shall use the estimate
$$
M_p(\mathcal C_{\mu, \beta}f,r)\lesssim \|f\|_{(p,\infty,\gamma)}\int_0^1\frac{d\mu(t)}{(1-rt)^{\beta+\gamma}}\approx \|f\|_{(p,\infty,\gamma)}\sum_{n=0}^\infty (n+1)^{\gamma+\beta-1}\mu_nr^n.
$$
We begin by considering the case $\beta<1/p'$.  Using  Theorem \ref{equivfmu} for $\alpha=\beta+\gamma-1/p'$, since $-1<\alpha<\gamma$ we have that $((n+1)^{\beta-1/q}\mu_n)\in \ell^q.$ 

Therefore, using \cite[Lemma 2.1]{B4},
\begin{align*}
\|\mathcal C_{\mu, \beta}f\|_{(p,q,\gamma)}^q&\lesssim \|f\|^q_{(p,\infty,\gamma)}\int_0^1(1-r)^{q\gamma-1}(\sum_{n=0}^\infty (n+1)^{\gamma+\beta-1}\mu_nr^n)^qdr\\
&\lesssim \|f\|_{(p,\infty,\gamma)}^q \sum_{n=0}^{\infty} 2^{-n\gamma q}(\sum_{k\in I_n} (k+1)^{\gamma+\beta-1}\mu_k)^q\\
&\lesssim \|f\|_{(p,\infty,\gamma)}^q\sum_{n=1}^\infty 2^{-n\gamma q}\mu_{2^{n-1}}^q(\sum_{k\in I_n} (k+1)^{\gamma+\beta-1})^q\\
&\lesssim \|f\|_{(p,\infty,\gamma)}^q\sum_{n=0}^\infty 2^{n\beta q}\mu_{2^{n}}^q\\
&\lesssim \|f\|_{(p,\infty,\gamma)}^q.
\end{align*}
Finally, assume that  $\beta=\frac{1}{p'}$, by the inclusions between the mixed norm spaces \eqref{inclu}, if $p_1>p$ we have that
\[
D_{\beta+\gamma-1/p'}F_\mu=D_{\gamma}F_\mu\in H(p,q,\gamma)\subset H(p_1,q,\gamma+\frac{1}{p}-\frac{1}{p_1}).
\]
Hence, using Lemma \ref{derivMNS}, $F_\mu\in H(p_1,q,\frac{1}{p}-\frac{1}{p_1})$. Therefore, by  Theorem \ref{equivfmu}, we get that $((n+1)^{1/p_1-1/p+1/p_1'-1/q}\mu_n)\in \ell^q$, which is equivalent to $((n+1)^{1/p'-1/q}\mu_n)\in \ell^q$ and arguing as above we obtain the result.

\end{proof}

\begin{cor}\label{corofinal2} Let $\mu$ be a positive Borel measure defined on $[0,1)$, $\gamma>0$,  $2\le p\le \infty$ and $0<q<\infty$. Then $\mathcal C_{\mu}$ maps $H(p,\infty,\gamma)$ into $H(p,q,\gamma)$ if and only if $(\mu_n (n+1)^{1-1/q})\in \ell^q$.
    
\end{cor}
\begin{proof} From Theorem \ref{teofinal} we know that  $\mathcal C_{\mu}:H(p,\infty,\gamma) \to H(p,q,\gamma)$ is bounded if and only if
$D_{\gamma+1/p}F_\mu \in H(p,q,\gamma)$. Let us see that this is equivalent to the fact that $(\mu_n (n+1)^{1-1/q})\in \ell^q$ for $p\ge 2$.

On the one hand, $D_{\gamma+1/p}F_\mu \in H(p,q,\gamma)\subset H(\infty,q,\gamma+1/p)$ and  from (ii) in Theorem \ref{equivfmu}
we have $(\mu_n (n+1)^{1-1/q})\in \ell^q$.

On the other hand, since $p\ge 2$, 
$$M_p(D_{\gamma+1/p}F_\mu,r)\lesssim (\sum_{n=0}^\infty \frac{\mu_n^p (n+1)^{\gamma p+1} r^{np}}{(n+1)^{2-p}})^{1/p}. $$
Therefore, arguing as in Theorem \ref{equivfmu}
\begin{align*}
    \|D_{\gamma+1/p}F_\mu\|^q_{(p,q,\gamma)}&\lesssim \int_0^1(1-r)^{\gamma q-1}
    (\sum_{n=0}^\infty \mu_n^p (n+1)^{(\gamma+1) p-1} r^{np})^{q/p}dr\\
    &\approx \|(\mu_n^p (n+1)^{ p-1})\|_{(1,q/p)}^{q/p},
\end{align*}
and using that $(\mu_n^p (n+1)^{ p-1})\in \ell(1,q/p)$ is equivalent to
$(\mu_n (n+1)^{ 1-1/q})\in \ell^q$ we get the result.
\end{proof}


\end{document}